\documentclass[11pt,oneside]{amsart}
\usepackage{enumerate,geometry,amssymb}
 \DeclareMathOperator{\tr}{tr} \DeclareMathOperator{\str}{str}
\DeclareMathOperator{\Tr}{Tr} \DeclareMathOperator{\Str}{Str}
\DeclareMathOperator{\inc}{inc}\DeclareMathOperator{\Inc}{Inc}
\DeclareMathOperator{\Id}{Id}
\newtheorem{prop}{Proposition}
\newtheorem{teo}{Theorem}
\begin{document}
\author{Paul A. Blaga}
\thanks{Partially supported by the Grant CEEX 130/2006 of the Romanian Ministry of Education and Research}
\dedicatory{Dedicated to Professor Nicolae Teleman, on his 65$^\text{th}$ birthday anniversary}
\address{University of Cluj-Napoca, Faculty of Mathematics and Computer Sciences\\
1, Kog\u{a}lniceanu Street\\
400084 Cluj-Napoca, Romania}
\email{pablaga@cs.ubbcluj.ro}
\title{The Dennis' Supertrace and the Hochschild Homology of Supermatrices}
\maketitle
\begin{abstract}
We construct, in this paper, a generalization of the Dennis trace (for matrices) to the case
of the supermatrices over an arbitrary (not necessarily commutative) superalgebra
with unit. By analogy with the ungraded case, we show how it is possible to use
this map to construct an isomorphism from the Hochschild homology of the
superalgebra to the Hochschild homology of the supermatrix algebra.
\end{abstract}
\section{The supertrace and the supercommutator}
We remind first a couple of things related to the $\mathbb{Z}_2$-grading of the
supermatrix algebra $M_{p,q}(R)$. For the general material regarding the superalgebra the reader should consult the classical books of Bartocci, Bruzzo and Hernandez-Ruiperez (\cite{bruzzo}) and Manin (\cite{manin}). A matrix $A\in M_{p,q}(R)$ is considered to be
homogeneous if it can be decomposed into blocks
\begin{equation}\label{one}
A=
\begin{pmatrix}
A_{11}&A_{12}\\ A_{21}&A_{22}
\end{pmatrix},
\end{equation}
where $A_{11}, A_{12}, A_{21}, A_{22}$ are matrices of type $(p,p), (p,q), (q,p), (q,q)$,
respectively and all the components of one of the matrix are homogeneous of the same
parity. Moreover, it is assumed that the elements of $A_{11}$ and $A_{22}$ have the same
parity and the same is true for the other pair of matrices. Now, a  matrix~(\ref{one})
satisfying this conditions is
\begin{itemize}
  \item \textit{even} if the elements of the diagonal blocks are even, while the elements
  from the blocks off the diagonal are odd;
  \item \textit{odd} if the elements of the diagonal blocks are odd, while the elements
  from the blocks off the diagonal are even.
\end{itemize}

Now, the \textit{supertrace} of matrices is the $R$-module morphism $\str:M_{p,q}(R)\to
R$, defined on homogeneous matrices~(\ref{one}) by
\begin{equation}\label{two}
\str(A)= \tr A_{11}+(-1)^{1+|A|}\tr A_{22}
\end{equation}
where $\tr$ is the ordinary trace of a matrix, while with  $|\;\cdot\;|$ it is denoted
the parity of an element. We shall use the same notation to denote the parity of an
element of the algebra $R$ and a matrix, because it will be always clear from the context
what kind of object we are dealing with.

It should be noted that, contrary to what one would expect, the supertrace coincides with
the ordinary trace of a matrix $A$ when the matrix $A$ is \textit{odd} and not even. On
the other hand, for matrices from $M_{p,0}(R)$ the supertrace is identical to the trace,
no matter what parity the matrices might have. Finally, we remark that $\str$ is, indeed,
an $R$-module morphism, in the sense that  not only it is linear, but also preserves the
parity.

Let, now, $A,B\in M_{p,q}(R)$  be two supermatrices. Their \textit{supercommutator} will
be
\begin{equation}\label{three}
\{A,B\}\doteq A\cdot B-(-1)^{|A|\cdot |B|}B\cdot A.
\end{equation}
Thus, if at least one of the two supermatrices is even, the supercommutator reduces to
the ordinary commutator of two matrices. There is a difference in sign only in the case
when both supermatrices are odd.

A very important property of the supercommutator, which will be useful also in the
following is that it is related rather ``nicely'' to the supertrace.

 Since the operations involved are either linear or bilinear, it is, clearly, enough to
make the computations on a system of  (homogeneous) generators of the algebra
$M_{p,q}(R)$. A very convenient such basis is constructed from matrices of the form
$E^{ij}(a)$, where $i,j\in\{1,\dots, p+q\}$, while $a$ is a homogeneous element of the
algebra $A$. Here
\begin{equation}\label{four}
E^{ij}_{kl}(a)=\delta_{kl}^{ij}\cdot a=
\begin{cases}
a&\text{if}\;\; i=k, j=l,\\ 0&\text{if}\;\; i\neq k\;\text{or}\; j\neq l.
\end{cases}
\end{equation}
The parity of a matrix of the form $E^{ij}(a)$ is related to the parity of the element
$a$ in a very simple manner: if $a$ is in the diagonal block components, than the two
objects have the same parity, otherwise their parity is opposed. More precisely,
\begin{equation}\label{five}
\left|E^{ij}(a)\right|=
\begin{cases}
|a|&\text{if}\quad i\leq p, j\leq p\quad\text{or}\quad i>p, j>p,\\ 1+|a|&\text{if}\quad
i\leq p, j>p\quad\text{or}\quad i>p, j\leq p.
\end{cases}
\end{equation}
It is  not difficult to see that the family of supermatrices
\begin{equation}\label{six}
\left\{E^{ij}(a)\;|\; 1\leq i,j\leq p+q, a\in R_0\;\text{or}\; a\in R_1 \right\}
\end{equation}
 is an
ideal in $M_{p,q}(R)$:
\begin{equation}\label{seven}
E^{ij}(a)\cdot E^{kl}(b)=
\begin{cases}
E^{il}(a\cdot b),&\text{if}\;k=l,\\ 0,&\text{if}\;k\neq l.
\end{cases}
\end{equation}
Now,
\begin{equation*}
\begin{split}
\left\{E^{ij}(a),E^{kl}(b)\right\}&=E^{ij}(a)\cdot
E^{kl}(b)-(-1)^{|E^{ij}(a)|\cdot|E^{kl}(b)|}E^{kl}(b)\cdot E^{ij}(a)=\\ &=
\begin{cases}
0,& j\neq k, i\neq l,\\ E^{il}(ab)&j=k, i\neq l,\\
-(-1)^{|E^{ij}(a)|+|E^{ki}(b)|}E^{kj}(ba),& j\neq k, i=l,\\
E^{ii}(ab)-(-1)^{|E^{ij}(a)|\cdot|E^{ji}(b)|}E^{jj}(ba),& j=k, i=l.
\end{cases}
\end{split}
\end{equation*}
We note first that the matrices $E^{ij}(a)$ with $i\neq j$ are commutators. Let us
compute now the supertrace of the supercommutator of two generating matrices, separately
for each combination of indices.

We get, obviously, $0$ if $j\neq k $ and $i\neq l$. If $j=k, i\neq l$, we obtain
\begin{equation}\label{eight}
\str \left\{E^{ij}(a),E^{jl}(b)\right\}=\str E^{il}(ab)=0.
\end{equation}
The same is true for the case $j\neq k, i=l$. The only interesting case is the last one.
Now we have
\begin{equation}\label{nine}
\str
\left\{E^{ij}(a),E^{ji}(b)\right\}=\str\left(E^{ii}(ab)-(-1)^{\left|E^{ij}(a)\right|\cdot\left|E^{ji}(b)\right|}E^{jj}(ba)\right)
\end{equation}
We have several subcases to consider here:
\begin{enumerate}[(i)]
\item Suppose we have $i=j$. In this case, we have $\left|E^{ii}(a)\right|=|a|$, $\left|E^{ii}(b)\right|=|b|$,
thus,
\begin{equation*}
\begin{split}
\left\{E^{ii}(a),E^{ii}(b)\right\}&=E^{ii}(ab)-(-1)^{|a|\cdot|b|}E^{ii}(ba)=E^{ii}(ab-(-1)^{|a|\cdot|b|}ba)\\
&=E^{ii}\left(\{a,b\}\right),
\end{split}
\end{equation*}
therefore
\begin{equation*}
\str \left\{E^{ii}(a),E^{ii}(b)\right\}=\str E^{ii}\left(\{a,b\}\right)=
\begin{cases}
\{a,b\},& i\leq p\\ (-1)^{1+|a|+|b|}\{a,b\},& i> p.
\end{cases}
\end{equation*}
\item $i<j\leq p$. In this case we have $\left|E^{ij}(a)\right|=|a|$, $\left|E^{ji}(b)\right|=|b|$, therefore
\begin{equation*}
\left\{E^{ij}(a),E^{ji}(b)\right\}=E^{ii}(ab)-(-1)^{|a|\cdot|b|}E^{jj}(ba),
\end{equation*}
hence
\begin{equation*}
\begin{split}
\str \left\{E^{ij}(a),E^{ji}(b)\right\}&=\str E^{ii}(ab)-(-1)^{|a|\cdot|b|}\str
E^{jj}(ba)=\\&=ab-(-1)^{|a|\cdot|b|}ba=\{a,b\}.
\end{split}
\end{equation*}
\item $i\leq p<j$. Now $\left|E^{ij}(a)\right|=1+|a|$, $\left|E^{ji}(b)\right|=1+|b|$, and then
\begin{equation*}
\left\{E^{ij}(a),E^{ji}(b)\right\}=E^{ii}(ab)-(-1)^{1+|a|+|b|+|a|\cdot|b|}E^{jj}(ba)
\end{equation*}
and
\begin{equation*}
\begin{split}
\str\left\{E^{ij}(a),E^{ji}(b)\right\}&=\str E^{ii}(ab)-(-1)^{1+|a|+|b|+|a|\cdot|b|}\str
E^{jj}(ba)=\\ &=ab-(-1)^{1+|a|+|b|+|a|\cdot|b|}\cdot(-1)^{1+|a|+|b|}ba=\\&=
ab-(-1)^{|a|\cdot |b|}ba=\{a,b\}.
\end{split}
\end{equation*}
\item $p<i<j$. In this situation, $\left|E^{ij}(a)\right|=|a|$, $\left|E^{ji}(b)\right|=|b|$, therefore
\begin{equation*}
\left\{E^{ij}(a),E^{ji}(b)\right\}=E^{ii}(ab)-(-1)^{|a|\cdot|b|}E^{jj}(ba),
\end{equation*}
but
\begin{equation*}
\begin{split}
\str \left\{E^{ij}(a),E^{ji}(b)\right\}&=\str E^{ii}(ab)-(-1)^{|a|\cdot|b|}\str
E^{jj}(ba)=\\&=(-1)^{1+|a|+|b|}\left(ab-(-1)^{|a|\cdot|b|}ba\right)=(-1)^{1+|a|+|b|}\{a,b\}.
\end{split}
\end{equation*}
\item $j<i\leq p$. This case is identical to the case (ii).
\item $j\leq p<i$. We have now, as in the case (iii), $|E^{ij}(a)|=1+|a|$,
$|E^{ji}(b)|=1+|b|$, therefore
\begin{equation*}
\left\{E^{ij}(a),E^{ji}(b)\right\}=E^{ii}(ab)-(-1)^{1+|a|+|b|+|a|\cdot|b|}E^{jj}(ba),
\end{equation*}
but
\begin{equation*}
\begin{split}
\str\left\{E^{ij}(a),E^{ji}(b)\right\}&=\str E^{ii}(ab)-(-1)^{1+|a|+|b|+|a|\cdot|b|}\str
E^{jj}(ba)=\\
&=(-1)^{1+|a|+|b|}ab-(-1)^{1+|a|+|b|+|a|\cdot|b|}ba=\\&=(-1)^{1+|a|+|b|}\{a,b\}.
\end{split}
\end{equation*}
\item $p<j<i$. This case is identical to the case (iv).
\end{enumerate}
\section{The Hochschild homology of superalgebras}
The Hochschild complex for superalgebras (Kassel, 1986), is very similar to the analogous
complex for ungraded case. Namely, the chain groups are, as in the classical case,
$C_m(R)=R^{\otimes m+1}$, where, of course, the tensor product should be understood in
the graded sense, while the face maps and degeneracies are given by
\begin{equation}
\delta^m_i(a_0\otimes\dots\otimes a_m)=a_0\otimes \dots \otimes a_i a_{i+1}\otimes\dots
a_n,\quad \text{if}\quad 0\leq i <m,
\end{equation}
\begin{equation}
\delta^m_m(a_0\otimes\dots\otimes a_m)=(-1)^{|a_m|(|a_0|+\dots+|a_{m-1})}a_ma+0\otimes
a_1\otimes\dots \otimes a_{m-1},
\end{equation}
\begin{equation}
s^m_i(a_0\otimes\dots\otimes a_m)=a_0\otimes\dots\otimes a_i\otimes 1\otimes
a_{i+1}\otimes\dots\otimes a_m,\quad 0\leq i\leq m.
\end{equation}
Now the differential is defined in the usual way, meaning $d^m:C_m(R)\to C_{m-1}(R)$,
\begin{equation}
d^m=\sum\limits_{i=0}^m(-1)^i\delta_i^m.
\end{equation}
and the Hochschild  homology of the superalgebra is just the homology of the complex
$(C(R),d)$. In particular, it is easy to see that for any superalgebra $R$ we have
\begin{equation}
H_0(R)=R/\{R,R\},
\end{equation}
where $\{R,R\}$ is the subspace generated by the supercommutators.
\section{The Dennis supertrace and its properties}
If $A^1$ and $A^2$ are two square matrices over an arbitrary algebra $R$, then their
product is
\begin{equation*}
\left(A^1\cdot A^2\right)_{ij}=\sum\limits_{k=1}^nA^1_{ik}\cdot A^2_{kj},
\end{equation*}
therefore, the trace of the product is
\begin{equation*}
\tr \left(A^1\cdot A^2\right)=\sum\limits_{i=1}^n\sum\limits_{k=1}^nA^1_{ik}\cdot
A^2_{ki}.
\end{equation*}
Completely analogously, the trace of the product of $m+1\geq 2$ matrices is
\begin{equation*}
\tr \left(A^0\cdot A^1\dots A^m\right)=\sum^{n}_{i_0=1}\sum A^0_{i_0i_1}\cdot
A^1_{i_1i_2}\dots A^{m-1}_{i_{m-1}i_m}A^m_{i_mi_0},
\end{equation*}
where the second sum is taken after all the possible values of the indices $i_1,\dots,
i_m\in \{1,\dots, n\}.$

Now, the very natural idea of Dennis was to define a generalized trace map (which is now
often called \textit{Dennis trace}),
\begin{equation*}
\Tr^m: M_n(R)^{\otimes m+1}\to R^{\otimes m+1},
\end{equation*}
putting
\begin{equation}\label{Tr}
\Tr^m \left(A^0\otimes A^1\otimes\dots \otimes A^m\right)=\sum^{n}_{i_0=1}\sum
A^0_{i_0i_1}\otimes A^1_{i_1i_2}\otimes \dots \otimes A^{m-1}_{i_{m-1}i_m}A^m_{i_mi_0},
\end{equation}
with the second summation sign having the same significance as above. Dennis used the
generalized trace to construct an isomorphism between the Hochschild homology of the
matrix algebra over an algebra $R$ and that of the algebra itself.

The Dennis' construction can be carried out also in the case of superalgebras if we
replace the trace with the supertrace and we pay attention to the signs.

Namely, it is easy to see that the supertrace of a product of two homogeneous
supermatrices of type $(p,q)$ over a superalgebra $R$
\begin{equation*}
\str \left(A^1\cdot A^2\right)=\sum\limits_{i=1}^p\sum\limits_{k=1}^{p+q}A^1_{ik}\cdot
A^2_{ki} +
(-1)^{1+|A^1|+|A^2|}\sum\limits_{i=p+1}^{p+q}\sum\limits_{k=1}^{p+q}A^1_{ik}\cdot
A^2_{ki},
\end{equation*}
while for $m+1$ supermatrices we have
\begin{equation*}
\begin{split}
\str \left(A^0\cdot A^1\dots A^m\right)&=\sum^{p}_{i_0=1}\sum A^0_{i_0i_1}\cdot
A^1_{i_1i_2}\dots A^{m-1}_{i_{m-1}i_m}A^m_{i_mi_0}+\\
&+(-1)^{1+|A^0|+\dots+|A^m|}\sum^{p+q}_{i_0=p+1}\sum A^0_{i_0i_1}\cdot A^1_{i_1i_2}\dots
A^{m-1}_{i_{m-1}i_m}A^m_{i_mi_0},
\end{split}
\end{equation*}
where, as above, the second sum in each term is taken after all the values of the indices
$i_1,\dots, i_m\in \{1,\dots, p+q\}$.

Now, it is clear that to have a consistent generalization of the Dennis trace for the
$\mathbb{Z}_2$-graded case we should put (for homogeneous supermatrices)
\begin{multline}\label{Str}
\Str^m \left(A^0\otimes A^1\otimes\dots\otimes A^m\right)=\sum^{p}_{i_0=1}\sum
A^0_{i_0i_1}\otimes A^1_{i_1i_2}\otimes\dots \otimes A^{m-1}_{i_{m-1}i_m}A^m_{i_mi_0}+\\
+(-1)^{1+|A^0|+\dots+|A^m|}\sum^{p+q}_{i_0=p+1}\sum A^0_{i_0i_1}\otimes
A^1_{i_1i_2}\otimes\dots \otimes A^{m-1}_{i_{m-1}i_m}A^m_{i_mi_0}.
\end{multline}
We shall call this generalized supertrace the \textit{Dennis supertrace map}.

It is convenient to work, as before, with the (homogeneous) generators $E^{ij}(a)$ of the
supermatrix algebra $M_{p,q}(R)$. The nice thing about them is that the Dennis supertace
can be written down very easily, for these generators, because, as one can see
immediately,
\begin{equation*}
\Str^m\left(E^{i_0j_0}(a_0)\otimes\dots\otimes E^{i_mj_m}(a_m)\right)\neq 0
\end{equation*}
if and only if we have
\begin{equation*}
j_0=i_1, j_1=i_2,\dots, j_{m-1}=i_m, j_m=i_0,
\end{equation*}
therefore we shall suppose all the time that these conditions are fulfilled. Now, it is
easy to see that
\begin{prop}
The Dennis supertrace can be written on the homogeneous generators as
\begin{equation}\label{Str1}
\Str^m\left(E^{i_0i_1}(a_0)\otimes\dots\otimes E^{i_mi_0}(a_m)\right)=
\begin{cases}
a_0\otimes\dots\otimes a_m, & i_0\leq p,\\
(-1)^{1+|a_0|+\dots+|a_m|}a_0\otimes\dots\otimes a_m,& i_0>p.
\end{cases}
\end{equation}
\end{prop}
\begin{proof}
Clearly, the only thing that calls for a justification is the fact that
\begin{equation*}
|E^{i_0i_1}(a_0)|+\dots+|E^{i_mi_0}(a_m)|=|a_0|+\dots+|a_m|.
\end{equation*}
But this follows immediately if we notice that
\begin{equation*}
E^{i_0i_1}(a_0)\cdot E^{i_1i_2}(a_1)\dots E^{i_mi_0}(a_m)=E^{i_0i_0}(a_0\cdot a_1\dots
a_m),
\end{equation*}
therefore, on the one hand
\begin{equation*}
|E^{i_0i_1}(a_0)\cdot E^{i_1i_2}(a_1)\dots
E^{i_mi_0}(a_m)|=|E^{i_0i_1}(a_0)|+\dots+|E^{i_mi_0}(a_m)|
\end{equation*}
and, on the other hand,
\begin{equation*}
|E^{i_0i_1}(a_0)\cdot E^{i_1i_2}(a_1)\dots E^{i_mi_0}(a_m)|=|E^{i_0i_0}(a_0\cdot a_1\dots
a_m)|=|a_0\cdot a_1\dots a_m|=|a_0|+\dots+|a_m|,
\end{equation*}
where, when we wrote the second equality, we took into account the fact that the only
non-vanishing element of the matrix $E^{i_0i_0}(a_0\cdot a_1\dots a_m)$ is on the
diagonal, therefore the parity of the matrix is equal to the parity of that element.
\end{proof}
\begin{teo}
The family of mappings $\left\{\Str^m:M_{p,q}(R)^{\otimes m+1}\to R^{\otimes
m+1}\right\}$ defines a chain morphism between the Hochschild complex of the algebra
$M_{p,q}(R)$ and the Hochschild complex of the ground algebra $R$.
\end{teo}
\begin{proof}
Clearly, each Dennis supertrace is a linear map. All we have to do is to show that the
Dennis traces commute with the face maps of the two Hochschild complexes, i.e. with the
operators $\delta^m_k$, $k=0,\dots, m$. Again, it is enough to verify for elements of the
form
\begin{equation*}
E^{i_0i_1}\otimes\dots\otimes E^{i_mi_0}.
\end{equation*}
We shall discuss first the case $k<m$.
 As we saw above, the Dennis supetrace, calculated on such an element is
\begin{equation*}
\Str^m\left(E^{i_0i_1}(a_0)\otimes\dots\otimes E^{i_mi_0}(a_m)\right)=
\begin{cases}
a_0\otimes\dots\otimes a_m, & i_0\leq p,\\
(-1)^{1+|a_0|+\dots+|a_m|}a_0\otimes\dots\otimes a_m,& i_0>p.
\end{cases}
\end{equation*}
Thus, we have
\begin{multline*}
\delta^m_k \Str^m\left(E^{i_0i_1}(a_0)\otimes\dots\otimes E^{i_mi_0}(a_m)\right)=\\
=
\begin{cases}
a_0\otimes\dots a_ka_{k+1}\otimes\dots\otimes a_m, & i_0\leq p,\\
(-1)^{1+|a_0|+\dots+|a_m|}a_0\otimes\dots a_ka_{k+1}\otimes\dots\otimes a_m,& i_0>p.
\end{cases}
\end{multline*}
On the other hand,
\begin{multline*}
\delta^m_k\left(E^{i_0i_1}(a_0)\otimes\dots\otimes E^{i_mi_0}(a_m)\right)=\\
E^{i_0i_1}(a_0)\otimes\dots E^{i_ki_{k+1}}(a_k)E^{i_{k+1}i_{k+2}}(a_{k+1})\otimes\dots
\otimes E^{i_mi_0}(a_m)=\\ E^{i_0i_1}(a_0)\otimes\dots
E^{i_ki_{k+2}}(a_ka_{k+1})\otimes\dots \otimes E^{i_mi_0}(a_m),
\end{multline*}
therefore
\begin{multline*}
\Str^{m-1}\delta^m_k\left(E^{i_0i_1}(a_0)\otimes\dots\otimes E^{i_mi_0}(a_m)\right)=\\
=
\begin{cases}
a_0\otimes\dots a_ka_{k+1}\otimes\dots\otimes a_m, & i_0\leq p,\\
(-1)^{1+|a_0|+\dots+|a_m|}a_0\otimes\dots a_ka_{k+1}\otimes\dots\otimes a_m,& i_0>p.
\end{cases}
=\\ =\delta^m_k \Str^m\left(E^{i_0i_1}(a_0)\otimes\dots\otimes E^{i_mi_0}(a_m)\right),
\end{multline*}
where we have use the fact that $|a_ka_{k+1}|=|a_k|+|a_{k+1}|$. The only case  that needs
extra work is the case $k=m$. In this case we have, on the one hand,
\begin{multline*}
\delta^m_m\Str^m\left(E^{i_0i_1}(a_0)\otimes\dots\otimes E^{i_mi_0}(a_m)\right)=\\
\begin{cases}
(-1)^{|a_m|\sum\limits_{i=0}^{m-1}|a_i|}a_0\otimes\dots a_ka_{k+1}\otimes\dots\otimes
a_m, & i_0\leq p,\\
(-1)^{1+\sum\limits_{i=0}^{m}|a_i|+|a_m|\sum\limits_{i=0}^{m-1}|a_i|}a_0\otimes\dots
a_ka_{k+1}\otimes\dots\otimes a_m,& i_0>p.
\end{cases}
\end{multline*}
On the other hand,
\begin{multline*}
\Str^{m-1}\delta^m_m\left(E^{i_0i_1}(a_0)\otimes\dots\otimes E^{i_mi_0}(a_m)\right)=\\
=\Str^{m-1}\left[(-1)^{\left|E^{i_mi_0}(a_m)\right|\sum\limits_{l=0}^{m-1}\left|E^{i_li_{l+1}}(a_l)\right|}E^{i_mi_1}(a_ma_0)\otimes\dots\otimes
E^{i_{m-1}i_m}(a_{m-1})\right]=\\
=
\begin{cases}
(-1)^{\left|E^{i_mi_0}(a_m)\right|\sum\limits_{l=0}^{m-1}\left|E^{i_li_{l+1}}(a_l)\right|}a_ma_0\otimes\dots
\otimes a_{m-1},& i_m\leq p,\\
(-1)^{1+\left|E^{i_mi_0}(a_m)\right|\sum\limits_{l=0}^{m-1}\left|E^{i_li_{l+1}}(a_l)\right|+\sum\limits_{k=0}^m|a_k|}a_ma_0\otimes\dots
\otimes a_{m-1},& i_m> p.
\end{cases}
\end{multline*}
Now, to prove that $\Str^{m-1}\delta^m_m=\delta^m_m\Str^m$, we have to consider several
cases.
\begin{enumerate}[(i)]
\item $i_0\leq p, i_m\leq p$. In this case we have to prove that
\begin{equation*}
|a_m|\sum\limits_{k=0}^{m-1}|a_k|=\left|E^{i_mi_0}(a_m)\right|\sum\limits_{l=0}^{m-1}\left|E^{i_li_{l+1}}(a_l)\right|.
\end{equation*}
But, since we have $i_0, i_m\leq p$, it follows that $|E^{i_mi_0}(a_m)|=|a_m|$ and we
have already seen that
\begin{equation*}
\sum\limits_{k=0}^{m-1}|a_k|=\sum\limits_{l=0}^{m-1}\left|E^{i_li_{l+1}}(a_l)\right|.
\end{equation*}
\item $i_0\leq p, i_m> p$. Now the identity we have to prove is
\begin{equation*}
|a_m|\sum\limits_{k=0}^{m-1}|a_k|=\underbrace{1+\sum\limits_{k=0}^{m}|a_k|+\left|E^{i_mi_0}(a_m)\right|\sum\limits_{l=0}^{m-1}
\left|E^{i_li_{l+1}}(a_l)\right|}_{RHS}.
\end{equation*}
In this case, $\left|E^{i_mi_0}(a_m)\right|=1+|a_m|$. Moreover, we have
\begin{equation*}
\sum\limits_{l=0}^{m-1}
\left|E^{i_li_{l+1}}(a_l)\right|=\sum\limits_{l=0}^{m}\left|E^{i_li_{l+1}}(a_l)\right|-\left|E^{i_mi_0}(a_m)\right|=\sum\limits_{k=0}^{m}|a_k|-1-|a_m|=
1+\sum\limits_{k=0}^{m-1}|a_k|.
\end{equation*}
It follows, therefore, that
\begin{multline*}
RHS=1+\sum\limits_{k=0}^{m}|a_k|+\big(1+|a_m|\big)\left(1+\sum\limits_{k=0}^{m-1}|a_k|\right)=\\
1+\sum\limits_{k=0}^{m}|a_k|+1+\sum\limits_{k=0}^{m-1}|a_k|+|a_m|+|a_m|\sum\limits_{k=0}^{m-1}|a_k|=
|a_m|\sum\limits_{k=0}^{m-1}|a_k|,
\end{multline*}
so the identity is proven.
\item $i_0>p, i_m\leq p$. We have to show that
\begin{equation*}
1+\sum\limits_{k=0}^{m}|a_k|+|a_m|\sum\limits_{k=0}^{m-1}|a_k|=
\underbrace{\left|E^{i_mi_0}(a_m)\right|\sum\limits_{l=0}^{m-1}\left|E^{i_li_{l+1}}(a_l)\right|}_{RHS}.
\end{equation*}
The same reasoning we did before ensure us that we have $\left|E^{i_mi_0}(a_m)\right|=1+|a_m|$ and
\begin{equation*}
\sum\limits_{l=0}^{m-1} \left|E^{i_li_{l+1}}(a_l)\right|= 1+\sum\limits_{k=0}^{m-1}|a_k|,
\end{equation*}
hence
\begin{equation*}
\begin{split}
RHS&=(1+|a_m|)\left(1+\sum\limits_{k=0}^{m-1}|a_k|\right)=1+\sum\limits_{k=0}^{m-1}
|a_k|+|a_m|\sum\limits_{k=0}^{m-1}|a_k|+|a_m|=\\
&=1+\sum\limits_{k=0}^{m}|a_k|+|a_m|\sum\limits_{k=0}^{m-1}|a_k|,
\end{split}
\end{equation*}
so we are done.
\item $i_0>p, i_m>p$. The required identity reads
\begin{equation*}
1+\sum\limits_{k=0}^{m}|a_k|+|a_m|\sum\limits_{k=0}^{m-1}|a_k|=
1+\sum\limits_{k=0}^{m-1}|a_k|+\left|E^{i_mi_0}(a_m)\right|\sum\limits_{l=0}^{m-1}\left|E^{i_li_{l+1}}(a_l)\right|
\end{equation*}
or, which is the same,
\begin{equation*}
|a_m|\sum\limits_{k=0}^{m-1}|a_k|=\left|E^{i_mi_0}(a_m)\right|\sum\limits_{l=0}^{m-1}\left|E^{i_li_{l+1}}(a_l)\right|,
\end{equation*}
which is obvious, since in this case $\left|E^{i_mi_0}(a_m)\right|=|a_m|$.
\end{enumerate}
\end{proof}

It is pretty clear that the Dennis trace maps are onto. In fact, we can consider the map
\begin{equation*}
\inc:R\to M_{p,q}(R)
\end{equation*}
given by $\inc (a)=E^{11}(a)$. This map, which is, obviously, a linear morphism
($\left|E^{11}(a)\right|=|a|$), can be extended, for each natural $m$, to a morphism
\begin{equation*}
\Inc^m:R^{\otimes m+1}\to M_{p,q}(R)^{\otimes m+1},
\end{equation*}
\begin{equation*}
\Inc^m(a_0\otimes\dots\otimes a_m)=E^{11}(a_0)\otimes\dots \otimes E^{11}(a_m).
\end{equation*}
It can be shown immediately that
\begin{prop}
The family of maps $\left\{\Inc^m:R^{\otimes m+1}\to M_{p,q}(R)^{\otimes m+1},\;\; m\in
\mathbb{N} \right\}$, is a chain map from the Hochschild complex of $R$ to the Hochschild
complex of $M_{p,q}(R)$, which is a splitting of the Dennis supertrace.
\end{prop}
\section{The Hochschild homology of $M_{p,q}(R)$}
$\Inc$ is a right inverse of the Dennis supertrace, but, obviously, it is not, also, a
right inverse, so the Hochschild \textit{complexes} of $M_{p,q} (R)$ and $R$ are not
isomorphical. We shall prove that, however, the supertrace induces an isomorphism in
homology. To prove this, it is enough to verify that $\Inc$ is a left quasi-inverse of
the supetrace, i.e.
\begin{teo}
There is a chain homotopy $h:C(M_{p,q}(R))\to C(M_{p,q}(R))$ such that
\begin{equation*}
d\circ h+h\circ d= \Id-\Inc\circ\Str.
\end{equation*}
\end{teo}
\begin{proof}
We shall define the homotopy exactly as in the classical (non-graded) case and we shall
check that it does the job equally in the supercase. Thus, let us consider
\begin{equation*}
h=\sum^{m}_{l=0}(-1)^lh_l:M_{p,q}(R)^{\otimes m+1}\to M_{p,q}(R)^{\otimes m+2},
\end{equation*}
with
\begin{multline*}
h_l\left(E^{i_0i_1}(a_0)\otimes\dots\otimes E^{i_mi_0}(a_m)\right)=\\
=E^{i_01}(a_0)\otimes E^{11}(a_1) \otimes\dots\otimes E^{11}(a_l)\otimes
E^{1i_{l+1}}(1)\otimes E^{i_{l+1}i_{l+2}}(a_{l+1})\otimes\dots \otimes E^{i_mi_0}(a_m)
\end{multline*}
Let us, verify, first, that it works for the particular case of $m=0$. We have
\begin{equation*}
h_0\left(E^{i_0j_0}(a_0)\right)=E^{i_01}(a_0)\otimes E^{1j_0}(1).
\end{equation*}
\begin{multline*}
d^1\circ h_0\left(E^{i_0j_0}(a_0)\right)= \delta^1_0\circ
h_0\left(E^{i_0j_0}(a_0)\right)-\delta^1_1\circ h_0\left(E^{i_0j_0}(a_0)\right)=\\
=\delta^1_0\left(E^{i_01}(a_0)\otimes
E^{1j_0}(1)\right)-\delta^1_1\left(E^{i_01}(a_0)\otimes E^{1j_0}(1)\right)=\\
=
\begin{cases}
E^{i_0j_0}(a_0)&i_0\neq j_0\\ E^{i_0i_0}(a_0)-(-1)^{\left|E^{i_01}(a_0)\right|\cdot
\left|E^{1i_0}(1)\right|}E^{11}(a_0)&i_0=j_0
\end{cases}=\\
=
\begin{cases}
E^{i_0j_0}(a_0)&i_0\neq j_0\\
 E^{i_0i_0}(a_0)-E^{11}(a_0)&i_0=j_0\leq p,\\
  E^{i_0i_0}(a_0)-(-1)^{1+|a_0|}E^{11}(a_0)&i_0=j_0>p
\end{cases}=\left(\Id-\Inc\circ\Str\right)\left(E^{i_0j_0}(a_0)\right).
\end{multline*}
Thus, the claim is true at the lowest level. Take now an arbitrary $m\in\mathbb{N}$. Let
us compute first $\delta^{m+1}_{m+1}\circ h_m$. We have
\begin{multline*}
\delta^{m+1}_{m+1}\circ h_m\left(E^{i_0i_1}(a_0)\otimes\dots\otimes
E^{i_mi_0}(a_m)\right)=\\ \delta^{m+1}_{m+1}\left(E^{i_01}(a_0)\otimes
E^{11}(a_1)\otimes\dots\otimes E^{11}(a_m)\otimes E^{1i_0}(1)\right)=\\
=
(-1)^{\left|E^{1i_0}(1)\right|\left(\left|E^{i_01}(a_0)\right|+\sum\limits^{m}_{k=1}|a_k|\right)}
E^{11}(a_0)\otimes E^{11}(a_1)\otimes\dots\otimes E^{11}(a_m)=\\ =
\begin{cases}
E^{11}(a_0)\otimes E^{11}(a_1)\otimes\dots\otimes E^{11}(a_m),& i_0\leq p,\\
(-1)^{1+\sum\limits^{m}_{k=0}|a_k|} E^{11}(a_0)\otimes E^{11}(a_1)\otimes\dots\otimes
E^{11}(a_m),& i_0>p.
\end{cases}
\end{multline*}
On the other hand,
\begin{equation*}
\Str^m\left(E^{i_0i_1}(a_0)\otimes\dots\otimes E^{i_mi_0}(a_m)\right)=
\begin{cases}
a_0\otimes\dots\otimes a_m, & i_0\leq p,\\
(-1)^{1+\sum\limits^m_{k=0}|a_k|}a_0\otimes\dots\otimes a_m,& i_0>p.
\end{cases}
\end{equation*}
It follows then, immediately, that
\begin{equation*}
\delta^{m+1}_{m+1}\circ h_m=\Inc^m\circ\Str^m.
\end{equation*}
Moreover, we have
\begin{equation*}
\begin{split}
\delta^{m+1}_0\circ h_0\left(E^{i_0i_1}(a_0)\otimes\dots\otimes E^{i_mi_0}(a_m)\right)=\\
=\delta^{m+1}_0\left(E^{i_01}(a_0)\otimes E^{1i_1}(1)\otimes
E^{i_1i_2}(a_1)\otimes\dots\otimes E^{i_mi_0}(a_m)\right)=\\
E^{i_0i_1}(a_0)\otimes\dots\otimes E^{i_mi_0}(a_m)=\\
=\Id\left(E^{i_0i_1}(a_0)\otimes\dots\otimes E^{i_mi_0}(a_m)\right).
\end{split}
\end{equation*}
Now, exactly as in the classical case, one verifies immediately that if $1\leq l\leq m$ then
\begin{equation*}
\delta^{m+1}_l\circ h_l=\delta^{m+1}_l\circ h_{l-1},
\end{equation*}
while, if $k<l\leq m$, then
\begin{equation*}
\delta^{m+1}_k\circ h_l= h_{l-1}\circ \delta^{m}_k
\end{equation*}
and, also, if $k\geq l$, then
\begin{equation*}
\delta_k^{m+1}\circ h_l=h_l\circ \delta ^m_{k-1}.
\end{equation*}
To summarize, we have the following set of relations:
\begin{align}
\delta_{m+1}^{m+1}\circ h_m&=\Inc^m\circ \Str^m;\label{rel1}\\
\delta_0^{m+1}\circ h_0&=\Id;\label{rel2}\\
\delta_l^{m+1}\circ h_l&=\delta_l^{m+1}\circ h_{l-1},\quad \text{if $1\leq l\leq m$};\label{rel3}\\
\delta^{m+1}_k\circ h_l&=h_{l-1}\circ\delta^m_k,\quad \text{if $k<l\leq m$};\label{rel4}\\
\delta_k^{m+1}\circ h_l&=h_l\circ\delta_{k-1}^m,\quad \text{if $k\geq l+2$}.\label{rel5}.
\end{align}
We have everything we need to prove our assertion:
\begin{equation*}
\begin{split}
d^{m+1}\circ h^m+h^m\circ
d^m&=\sum^{m+1}_{k=0}\sum_{l=0}^m(-1)^{k+l}\delta^{m+1}_k\circ h^m_l+\sum
_{k=0}^m\sum _{l=0}^{m-1}(-1)^{k+l}h^m_l\circ \delta^m_k=\\
&=\delta_0^{m+1}\circ h_0-\delta_{m+1}^{m+1}\circ h_m +\sum\limits_{l=1}^m\delta_l^{m+1}\circ h_l-\sum\limits_{l=1}^m\delta_l^{m+1}\circ h_{l-1}+\\
&+\sum\limits_{k=0}^{l-1}\sum_{l=1}^m(-1)^{k+l}\delta_k^{m+1}\circ h_l+\sum\limits_{k=l+2}^{m+1}\sum_{l=0}^{m-1}(-1)^{k+l}\delta_k^{m+1}\circ h_{l-1}+\\
&+\sum
_{k=0}^m\sum _{l=0}^{m-1}(-1)^{k+l}h^m_l\circ \delta^m_k=\Id-\Inc^m\circ \Str^m+\\
&+\sum\limits_{k=0}^{l-1}\sum_{l=1}^m(-1)^{k+l}h_{l-1}\circ \delta_k^{m}+\sum\limits_{k=l+2}^{m+1}\sum_{l=0}^{m-1}(-1)^{k+l}h_l\circ\delta_{k-1}^{m}+\\
&+\sum
_{k=0}^m\sum _{l=0}^{m-1}(-1)^{k+l}h^m_l\circ \delta^m_k=\Id-\Inc^m\circ \Str^m\\
&-\sum\limits_{k=0}^{p}\sum_{p=0}^{m-1}(-1)^{k+p}h_{p}\circ \delta_k^{m}-\sum\limits_{k=p+1}^{m}\sum_{p=0}^{m-1}(-1)^{k+p}h_p\circ\delta_{k}^{m}+\\
&+\sum
_{k=0}^m\sum _{l=0}^{m-1}(-1)^{k+l}h^m_l\circ \delta^m_k=\Id-\Inc^m\circ \Str^m-\\
&-\sum
_{k=0}^m\sum _{p=0}^{m-1}(-1)^{k+p}h^m_p\circ \delta^m_k+\sum
_{k=0}^m\sum _{l=0}^{m-1}(-1)^{k+l}h^m_l\circ \delta^m_k=\\
&=\Id-\Inc^m\circ \Str^m,
\end{split}
\end{equation*}
where we used the relations~(\ref{rel1})~--~(\ref{rel5}).
Thus, we have a quasi-isomorphisms between the two chain complexes, which means
that the two Hochschild homologies are isomorphic.
\end{proof}
\section{Final remarks} The basic ideas of these proof are ``super''-versions of the classical, ungraded proof (see \cite{rosenberg}). They amount to an unpublished result of R.K. Dennis (whence the name). We notice, however, to avoid confusions, that the term ``Dennis trace'' is also used for another map (related to the generalized trace, also introduced by Dennis), establishing a connection between  algebraic K-theory and Hochschild homology (see~\cite{rosenberg}).

The Dennis supertraces can be used, as well, to provide a proof of the Morita invariance of the cyclic homology of the superalgebras (see~\cite{blaga2}). We also managed to prove, recently, the general Morita invariance of Hochschild homology of superalgebras, not only for the case of supermatrices (see~\cite{blaga1}). We used there a spectral sequence argument. Probably the more ``economical'' tools used by McCarthy (\cite{mccarthy}, see also the book of Loday~\cite{loday}) can be adapted, as well, to the super-case.

\section*{Acknowledgments} This paper was started when the author was a postdoc at the Universit\'a Politecnica delle Marche from Ancona, in the framework of the European Research Training Network "Geometric Analysis". The author would like to acknowledge both the financial support of the European Commission and the human and scientific support of Professor Nicolae Teleman.

\end{document}